\documentclass{amsart}
\usepackage{amsfonts,amssymb,amscd,amsmath,enumerate,verbatim,calc}

\newcommand{\CM}{Cohen-Macaulay}

\newcommand{\n}{\mathfrak{n} }
\newcommand{\m}{\mathfrak{m} }

\newcommand{\rt}{\rightarrow}

\newcommand{\image}{\operatorname{image}}

\newcommand{\depth}{\operatorname{depth}}
\newcommand{\projdim}{\operatorname{projdim}}
\newcommand{\injdim}{\operatorname{injdim}}

\newcommand{\Tor}{\operatorname{Tor}}

\newcommand{\curv}{\operatorname{curv}}
\newcommand{\injcurv}{\operatorname{injcurv}}

\newcommand{\Hom}{\operatorname{Hom}}
\newcommand{\Ext}{\operatorname{Ext}}

\theoremstyle{plain}

\newtheorem{theorem}{Theorem}[section]
\newtheorem{corollary}[theorem]{Corollary}
\newtheorem{lemma}[theorem]{Lemma}

\theoremstyle{definition}

\newtheorem{s}[theorem]{}
\newtheorem{remark}[theorem]{Remark}
\newtheorem{example}[theorem]{Example}

\theoremstyle{remark}
\begin{document}

\title[Obstructions]{Obstructions to curvature of modules over Cohen-Macaulay rings}
\author{Tony~J.~Puthenpurakal}
\date{\today}
\address{Department of Mathematics, IIT Bombay, Powai, Mumbai 400 076, India}

\email{tputhen@gmail.com}

 \begin{abstract}
Let $(A,\m)$ be a \CM \ local ring with residue field $k$. If $M$ is a finitely generated $A$-module then set $\curv_A(M) = \limsup_n\sqrt[n]{\beta_n^A(M)}$. We show that under mild hypotheses the existence of a single module $M$ with $1 \leq \curv(M) < \curv(k)$ imposes obstructions to both $\curv(k)$ and $\curv(M)$.
Similarly we show that the condition $\Tor^A_n(M, N) = 0$ for $n \gg 0$ imposes constraints on both $\curv(M)$ and $\curv(N)$.
\end{abstract}
 \maketitle
\section{introduction}
Let $(A,\m)$ be a Noetherian local ring with residue field $k$ and let $M$ be a finitely generated $A$-module. By $\ell(-)$ we denote the length function on $A$-modules. Let $\beta_n^A(M) = \ell(\Tor^A_n(M,k))$ for $n \geq 0$ be the $n^{th}$-betti number of $M$.
We set the \emph{curvature} of $M$ to be
\[
\curv_A(M) = \limsup_n\sqrt[n]{\beta_n^A(M)}.
\]
We drop the subscript $A$ if it is clear from the context.
It can be shown that $\curv(M) \leq \curv(k) $, see  \cite[4.1.9]{A} and $\curv(k) < \infty$; see \cite[4.1.5]{A}.  It can be shown that $A$ is \emph{not} a complete intersection if and only if $\curv(k) > 1$; see \cite[8.2.1]{A}.

In this paper we study the set $c(A) = \{ \curv M \mid \projdim_A M = \infty \}$ when $A$ is \CM \ and \emph{not} a complete intersection. We note that $c(A) \subseteq [1, \curv k]$. It is not known whether $c(A)$ is a finite set.  We show that if $c(A)$ is not a singleton set then under mild hypotheses it yields an obstruction of $\curv(k)$. We also show that if $\Tor^A_n(M, N) = 0$ for $n \gg 0$ then it imposes a constraint on $\curv(M)$ and $\curv(N)$.

\s\label{setup} \emph{Setup:} Let $(A,\m)$ be a \CM \ local ring of dimension $d$. Further  assume $A$ is \emph{NOT} a complete intersection. Let
$$e(A) = \lim_{n \rt \infty} \frac{d!}{n^d} \ell(A/\m^n) \quad \text{be the multiplicity of $A$}. $$
All modules considered will be finitely generated.

Our first result is
\begin{theorem}\label{first}
(with hypotheses as in \ref{setup}). Further assume that $\lim_{n \rt \infty} \sqrt[n]{\beta_n(k)}$ exists. Suppose there exists an $A$-module $M$ with $\projdim_A M = \infty$ and $\curv(M) < \curv(k)$.
Then
\begin{enumerate}[\rm (1)]
  \item $\displaystyle{\curv(k) \leq \frac{e(A)}{2} - 1.}$
  \item $\curv(M) < \sqrt{e(A)}  - 1$.
\end{enumerate}
\end{theorem}
\begin{remark}
All known examples of modules satisfy $$\limsup_n\sqrt[n]{\beta_n(M)} = \liminf_n\sqrt[n]{\beta_n(M)}$$ and so in all known cases  the limit exists. See \cite[4.3.6]{A}.
\end{remark}

\s We note that for any \CM \ local ring $A$ it is known that $\curv(k) \leq e(A) - 1$; see \cite[4.2.7]{A}.  The bound is attained for the so called rings of minimal multiplicity, see \cite[p.\ 50]{A}.

As an easy consequence of Theorem \ref{first} we obtain
\begin{corollary}\label{cor-first}
  (with hypotheses as in \ref{first}). Assume $\curv(k) > e(A)/2 - 1$.
If $\projdim M = \infty$ then $\curv(M) = \curv(k)$.
\end{corollary}

\s To study infinite injective resolutions the notion of \emph{injective curvature} was introduced in the paper \cite{A-ext}.
We define
$$ \injcurv(M) = \limsup_n \sqrt[n]{\ell(\Ext^n_A(k, M))}.$$

Our next result is also unexpected.
\begin{theorem}
\label{second}(with hypotheses as in \ref{setup}).  Let $M, N$ be $A$-modules.
\begin{enumerate}[\rm (1)]
  \item If $\Tor^A_i(M, N) = 0$ for $i \gg 0$ then
  \begin{enumerate}[\rm (a)]
    \item $$\min \{ \curv{M}, \curv{N} \} \leq \sqrt{e(A)} - 1. $$
    \item Further assume $\projdim M = \projdim N  = \infty$. Then
    $$\max \{ \curv{M}, \curv{N} \} \leq \frac{e(A)}{2} - 1. $$
  \end{enumerate}
  \item If $\Ext^i_A(M, N) = 0$ for $i \gg 0$ then
  \begin{enumerate}[\rm (a)]
    \item $$\min \{ \curv{M}, \injcurv{N} \} \leq \sqrt{e(A)} - 1. $$
    \item Further assume $\projdim M = \injdim N  = \infty$. Then
    $$\max \{ \curv{M}, \injcurv{N} \} \leq \frac{e(A)}{2} - 1. $$
  \end{enumerate}
\end{enumerate}
\end{theorem}
As an easy consequence of Theorem \ref{second} we obtain
\begin{corollary}\label{cor-first-s}
  (with hypotheses as in \ref{first}). Assume $\curv(k) > e(A)/2 - 1$.
  \begin{enumerate}[\rm (a)]
    \item If $\Tor^A_n(M, N) = 0$ for all $n \gg 0$ then $\projdim_A M < \infty$ or $\projdim_A N < \infty$.
    \item If $\Ext^n_A(M, N) = 0$ for $n \gg 0$ then  $\projdim_A M < \infty$ or $\injdim_A N < \infty$.
  \end{enumerate}
\end{corollary}

\begin{example}
Let $(A,\m)$ be a Gorenstein local ring with $e(A) = \mu(\m) - \dim A + 2$. If $h = \mu(\m) - \dim A \geq 3$ then it is not difficult to prove that $\curv(k) > e(A)/2 - 1$, see \cite[Theorem 2]{S}. It follows that $A$ is the so called AB-ring, i.e., $\Ext^n_A(M, N) = 0$ for $n \gg 0$ if and only if $\Ext^n_A(N, M) = 0$ for $n \gg 0$. This fact was first observed in \cite{HJ}
\end{example}
\s We note that to compute curvature of a module $M$ we need to have knowledge of the asymptotic behaviour of betti numbers of a module.
So the following result is also un-expected.
\begin{theorem}\label{third}(with hypotheses as in \ref{setup}).
 Let $i_0 \geq \dim A - \depth M $.
If $\projdim M = \infty$ and
$$ \frac{\beta_{i_0+1}(M)}{\beta_{i_0}(M)}  > \frac{e(A)}{1 + \curv(k)}  - 1,$$
then $\curv(M) \geq \liminf_n \sqrt[n]{\beta_n(k)}.$
\end{theorem}
\begin{remark}
We note that $\liminf_n \sqrt[n]{\beta_n(k)} > 1$, see \cite[8.2.1]{A}.
\end{remark}

We now describe in brief the contents of this paper. In section two we discuss some preliminaries. In section three we prove Theorem \ref{first} and Corollary \ref{cor-first}. In section four we prove Theorem \ref{second}. Finally in section five we prove Theorem \ref{third}.

\begin{remark}
  The main contribution of this paper is to guess the results. The proofs are not particularly hard.
\end{remark}
\section{Preliminaries}
In this section we discuss a few preliminary results that we need.

\s \label{base} For most of our results we need that the residue field of $A$ is infinite. For one particular application we need to assume that $A$ is complete. Both these base changes are of  the type $(A, \m) \rt (B, \n)$ flat with $\m B = \n$. If the residue field of $A$ is finite then we take $B = A[X]_{\m A[X]}$. We note that the residue field of $B$ is $k(X)$ which is infinite.
If $A$ is not complete then set $B$ to be the completion of $A$.

Let $(A, \m) \rt (B, \n)$ be flat with $\m B = \n$. Set $l = B/\n$. If $M$ is an $A$-module then set $M^\prime = M \otimes_A B$. We note that $k^\prime = l$. We have the following properties which are well-knowm:
\begin{enumerate}
  \item $\dim M = \dim M^\prime$ and $\depth_A M = \depth_B M^\prime$. In particular $A$ is \CM \ if and only if $B$ is \CM.
  \item $\beta_n^A(M) = \beta_n^B(M^\prime)$ for all $n \geq 0$.
  \item $\Tor^A_n(M, N)\otimes_A B = \Tor^B_n(M^\prime, N^\prime)$ for all $n \geq 0$. As $B$ is faithfully flat over $A$ we obtain $\Tor^A_n(M, N) = 0$ if and only if
  $\Tor^B_n(M^\prime, N^\prime) = 0$.
  \item $\Ext_A^n(M, N)\otimes_A B = \Ext_B^n(M^\prime, N^\prime)$ for all $n \geq 0$. As $B$ is faithfully flat over $A$ we obtain $\Ext_A^n(M, N) = 0$ if and only if
  $\Ext_B^n(M^\prime, N^\prime) = 0$.
  \item $\curv_A(M) = \curv_B(M^\prime)$ and $\injcurv_A (M) = \injcurv_B (M^\prime)$.
  \item $e(A) = e(B)$.
\end{enumerate}

\s\label{sup} The reason why we need the \CM \ assumption is that we can reduce to dimension zero easily. To enable this effectively we need the notion of superficial elements.
An element $x \in \m$ is said to be $A$-superficial if there exists non-negative integers $c, n_0$ such that $(\m^{n+1} \colon x)\cap \m^c = \m^n$ for all $n \geq n_0$. It is well-known that superficial elements exist when the residue field of $A$ is infinite. If $\depth A > 0$ and $x$ is $A$-superficial then $x$ is $A$-regular and $(\m^{n+1}\colon x) = \m^n$ for all $n \gg 0$.
Furthermore $e(A) = e(A/(x))$, cf. \cite[Corollary 10]{P-th}.  We note that if $\dim A > 0$ and $x$ is $A$-superficial then $x \in \m \setminus \m^2$.

We need the following result.
\begin{lemma}
\label{mod-x} Let $(A,\m)$ be a Noetherian local ring and let $x \in \m \setminus \m^2$ be $A$-regular. Set $k = A/\m$ and $B = A/(x)$ Then
\begin{enumerate}[\rm (1)]
  \item $\limsup_n\sqrt[n]{\beta_n^A(k)} = \limsup_n\sqrt[n]{\beta_n^B(k)}$.
  \item $\liminf_n\sqrt[n]{\beta_n^A(k)} = \liminf_n\sqrt[n]{\beta_n^B(k)}$.
  \item If $\lim_{n \rt \infty}\sqrt[n]{\beta_n^A(k)}$ exists then so does $\lim_{n \rt \infty}\sqrt[n]{\beta_n^B(k)}$ and both are equal.
\end{enumerate}
\end{lemma}
\begin{proof}
(1) and (2): We have by \cite[3.3.5]{A},  $\beta_{n}^B(k) + \beta_{n-1}^B(k) = \beta_n^A(k)$ for all $n \geq 1$.
In particular  $\beta_{n}^B(k) \leq \beta_n^A(k)$. So we have

 $\limsup_n\sqrt[n]{\beta_n^A(k)} \geq  \limsup_n\sqrt[n]{\beta_n^B(k)}$ and
  $\liminf_n\sqrt[n]{\beta_n^A(k)} \geq \liminf_n\sqrt[n]{\beta_n^B(k)}$.

  By \cite[Corollary 3]{G}, we have $\beta_{n}^B(k) \geq \beta_{n-1}^B(k) $ for all $n \geq 1$. So
  $\beta_n^A(k) \leq 2 \beta_{n}^B(k)$. As $\lim_n \sqrt[n]{2} = 1$ it follows that

$\limsup_n\sqrt[n]{\beta_n^A(k)} \leq  \limsup_n\sqrt[n]{\beta_n^B(k)}$ and
  $\liminf_n\sqrt[n]{\beta_n^A(k)} \leq \liminf_n\sqrt[n]{\beta_n^B(k)}$.

  The result follows.

  (3) This follows from (1) and (2).
\end{proof}

\s \label{lim} We will need the following easily proven fact. Let $\{ a_n \}_{n \geq 0}$ and $\{ b_n \}_{n \geq 0}$ be sequences. Assume $a_n > 0$ and $b_n > 0$ for all $n$.
If $\limsup_n\sqrt[n]{a_n} < \liminf_n\sqrt[n]{b_n}$ then $\lim_{n \rt \infty} a_n/b_n = 0$.

\s \label{rudin} We will also need the following, see \cite[3.37]{RW}. Let $\{ c_n \}_{n \geq 0}$ be a sequence of positive numbers. Then
\[
\liminf_n \frac{c_{n+1}}{c_n} \leq \liminf_n \sqrt[n]{c_n}  \quad \text{and} \quad  \limsup_n \sqrt[n]{c_n} \leq \limsup_n \frac{c_{n+1}}{c_n}.
\]

\s By $\Omega^i(M)$ we denote the $i^{th}$-syzygy of $M$. Also $\mu(M)$ denotes the number of minimal generators of $M$.

\s \label{est} Let $U, V$ be finite length modules. Set $r(V) = \ell(\Hom_A(k, N))$. Then
\begin{enumerate}
  \item $\ell(U\otimes V) \geq \mu(U)\mu(V)$.
  \item $\ell(\Hom_A(U, V)) \supseteq \mu(U)r(V)$.
\end{enumerate}
 To see (1) we note that we have a surjective map $U\otimes V \rt U\otimes V \otimes k$. The result follows.

 (2)  We have a surjection $U \rt U/\m U$. This creates an inclusion \\ $\Hom_A(U/\m U, V) \rt \Hom_A(U, V)$. The result follows.

\section{proof of Theorem \ref{first} and Corollary \ref{cor-first}}
We first give:
\begin{proof}[Proof of Theorem \ref{first}]
We first note that we may assume that the residue field of $A$ is infinite, see \ref{base}. After taking appropriate syzygy we may also assume that $M$ is maximal \CM. We reduce to the case when $\dim A = 0$ as follows. If $\dim A > 0$ then as the residue field of $A$ is infinite there exists an $A$-superficial element $x$. As $A$ is \CM \ we have $\depth A = \dim A > 0$. So $x$ is $A$-regular.
As $M$ is maximal \CM \ it is also $M$-regular. Set $B = A/(x)$.  We note that if $\mathbf{F}  $ is a minimal free $A$-resolution of $M$ then $\mathbf{F}\otimes_A B  $ is a minimal free $B$-resolution of $M/xM$.  So $\curv_A(M) = \curv_B(M/xM)$. We note that $x \in \m \setminus \m^2$. So by \ref{mod-x} we have $\lim_{n \rt \infty}sqrt[n]{\beta_n^B(k)}$  exists and is equal to  $\lim_{n \rt \infty}\sqrt[n]{\beta_n^A(k)}$. Thus it suffices to prove the result for $B$. Iterating we may reduce to the case when $\dim A = 0$.

Set $\curv(k) = \alpha$ and let $\curv(M) = \theta $. We have $\theta < \alpha$ and $\lim \sqrt[n]{\beta_n^A(k)} $ exists (and so is equal to $\alpha$). So by \ref{lim} we have
$\lim \beta_n(M)/\beta_n(k) = 0$. Set $X_i = \Omega^i(k)$.

We have an exact sequence $0 \rt \Omega^1(M) \rt A^{\beta_0(M)} \rt M \rt 0$.  Tensoring this exact sequence with $X_i$ we obtain
\[
0 \rt k^{\beta_{i+1}(M)} \rt \Omega^1(M)\otimes X_i  \rt X_i^{\beta_0(M)} \rt M \otimes X_i \rt 0.
\]
So we have by \ref{est}
\[
\beta_{i+}(M) + \beta_0(M)\ell(X_i)  \geq \beta_i(k) (\beta_0(M) + \beta_1(M)).
\]
So
\[
\ell(X_i) \geq \frac{1}{\beta_0(M)}\left( - \beta_{i+1}(M)  + \beta_i(k) (\beta_0(M) + \beta_1(M))\right).
\]
Similarly we obtain
\[
\ell(X_{i+1}) \geq \frac{1}{\beta_0(M)}\left( - \beta_{i+2}(M)  + \beta_{i+1}(k) (\beta_0(M) + \beta_1(M))\right).
\]
We note that $\ell(X_i) + \ell(X_{i+1})  = e(A)\beta_i(k)$. Adding the above two inequalities and dividing by $\beta_i = \beta_i(k)$ we obtain
\[
e(A) \geq \frac{1}{\beta_0(M)}\left( - (\beta_{i+1}(M) + \beta_{i+2}(M))/\beta_i + (\beta_0(M) +\beta_1(M)) (1 + \beta_{i+1}/\beta_i) \right).
\]
We take limsup on both sides. By our hypotheses and \ref{lim} we have
$$ \limsup_i - (\beta_{i+1}(M) + \beta_{i+2}(M))/\beta_i = 0.$$
By \ref{rudin} we have $\limsup_i \beta_{i+1}/\beta_i \geq \alpha$.
So we obtain
$$ e(A) \geq (1 + \beta_1(M)/\beta_0(M))(1+ \alpha).$$
Replacing $M$ by $\Omega^i(M)$ we obtain
$$ e(A) \geq (1 + \beta_{i+1}(M)/\beta_i(M))(1+ \alpha).$$
Taking limsup and by \ref{rudin} we obtain
\begin{equation*}
e(A) \geq (1 + \theta)(1 + \alpha). \tag{*}
\end{equation*}
(1) We note that $\theta \geq 1$. So we have by (*)
\[
(1 + \alpha) \leq e_0(A)/2.
\]
The result follows.

(2) As $\alpha > \theta$, by (*) we get $e(A) > (1+\theta)^2$. The result follows.
\end{proof}

We now give
\begin{proof}[Proof of Corollary \ref{cor-first}]
This follows  from part (1) of Theorem \ref{first}.
\end{proof}

\begin{example}
 \label{ex-1} We give an example  which shows that the assertion of Corollary \ref{cor-first} is sharp. Let $(B, \n)$ be a one dimensional \CM \ local ring of minimal multiplicity with $e(B) = h + 1$ with $h \geq 2$. Also assume that the residue field of $B$ is infinite. Note $B$ is not Gorenstein. Also, $G(B) = \bigoplus_{n \geq 0}\n^n/\n^{n+1}$, the associated graded ring of $B$ is \CM, see \cite{S-min}.  We also note that $\curv_B(k) = h$. Let $x \in \n$ be $B$-superficial. Then the initial form $x^*$ of $x$ is $G(B)$-regular, cf. \cite[Theorem 8]{P-th}. We let $(A, \m)  = (B/(x^2), \n/(x^2))$. Then $e(A) = 2e(B) = 2h + 2$.
It is not difficult to prove that $\lim_{\n \rt \infty} \beta_{n+1}^A(k)/\beta_n^A(k) = h$. So by \ref{rudin} we obtain that $\lim_{\n \rt \infty} \sqrt[n]{\beta_n^A(k)}$ exists and is equal to $h$.
We note that $h = e(A)/2 - 1$.  Note $M = A/(x)$ is a $1$-periodic $A$-module. So $\curv_A(M) = 1$.
\end{example}
\section{Proof of Theorem \ref{second}}
We first give
\begin{proof}[Proof of Theorem \ref{second}]
(1) We assume $\projdim M = \projdim N = \infty$. We first reduce to the case when $\dim A = 0$. Suppose $\dim A > 0$.
By Lemma \ref{base} we can assume that the residue field of $A$ is infinite. After taking sufficient high syzygies we can assume that $M$ and $N$ are both maximal \CM \ $A$-modules. Let $x \in \m$ be $A$-superficial.  Then $x$ is also $A$-regular. Set $B = A/(x)$. We note that $e(B) = e(A)$. As $M$ and $N$ are MCM we get that $x$ is both $M$ and $N$-regular.
The exact sequence $0 \rt N \rt N \rt N/xN \rt 0$ yields $\Tor^A_i(M, N/xN) = 0$ for $i \gg 0$. We have $\Tor^B_i(M/xM, N/xN) = \Tor^A_i(M, N/xN)$ for all $i \geq 0$,
 see \cite[Lemma 2, section 18]{M}. So $\Tor^B_i(M/xM, N/xN) =  0$ for $i \gg 0$. We also note that if $\mathbf{F}$ (respectively $\mathbf{G}$) is a minimal free $A$ resolution of $M$ (respectively $N$) then
$\mathbf{F} \otimes_A B$ (respectively $\mathbf{G}\otimes_A B$) is a minimal free $B$ resolution of $M/xM$ (respectively $N/xN$). So $\curv_B(M/xM) = \curv_A(M)$ and $\curv_B(N/xN) = \curv_A(N)$.
Thus it suffices to prove the result for $B$. Iterating we can reduce to the case when $\dim A = 0$. After taking a high syzygy of $M$ we may assume that $\Tor^A_i(M, N) = 0$ for $i > 0$.

We have an exact sequence $0 \rt \Omega^1(N) \rt A^{\beta_0(N)} \rt N \rt 0$. Set $M_j = \Omega^j(M)$. As $\Tor^A_i(M, N) = 0$ for all $i > 0$ we obtain after tensoring with $M_n$ an exact sequence $0 \rt M_n \otimes\Omega^1(N) \rt M_n^{\beta_0(N)} \rt M_n \otimes N \rt 0$. So we have by \ref{est}
\[
\beta_0(N) \ell(M_n) \geq (\beta_0(N) + \beta_1(N)) \beta_n(M).
\]
So we have
\[
\ell(M_n) \geq \left(1+ \frac{\beta_1(N)}{\beta_0(N)}\right) \beta_n(M).
\]
Similarly we obtain
\[
\ell(M_{n+1}) \geq \left(1+ \frac{\beta_1(N)}{\beta_0(N)}\right) \beta_{n + 1}(M).
\]
We note that $\ell(M_n) + \ell(M_{n+1}) = e(A) \beta_n(M)$. Adding the two equations and dividing by $\beta_n(M)$ we obtain
\[
e(A) \geq \left(1+ \frac{\beta_1(N)}{\beta_0(N)}\right) \left( 1 + \frac{\beta_{n+1}(M)}{\beta_{n}(M)}\right).
\]
Taking limsup and by \ref{rudin} we obtain
\[
e(A) \geq \left(1+ \frac{\beta_1(N)}{\beta_0(N)}\right) \left( 1 + \curv(M)\right).
\]
Replacing $N$ with $\Omega^n(N)$ we obtain
\[
e(A) \geq \left(1+ \frac{\beta_{n+1}(N)}{\beta_n(N)}\right) \left( 1 + \curv(M)\right).
\]
Taking limsup and by \ref{rudin} we obtain
\begin{equation*}
  e(A) \geq \left(1+ \curv(N)\right) \left( 1 + \curv(M)\right). \tag{*}
\end{equation*}

(a) If $\projdim M $ or $\projdim N$ is finite then we have nothing to prove. So assume $\projdim M = \projdim N = \infty$.
If $\min\{ \curv{M}, \curv(N) \} > \sqrt{e(A)} -1$ then by (*) we obtain $e(A) > e(A)$. This is a contradiction. The result follows.

(b) As $\projdim M = \projdim N = \infty$ we have $\curv(M) \geq 1$ and $\curv(N) \geq 1$. The result follows easily from (*).

(2)  We assume $\projdim M = \injdim N = \infty$.  We first reduce to the case when $\dim A = 0$. Suppose $\dim A > 0$.
By Lemma \ref{base} we can assume that the residue field of $A$ is infinite and that $A$ is complete. So $A$ has a canonical module $\omega$. After taking sufficient high syzygies we can assume that $M$ is a maximal \CM \ $A$-module. Let $ 0 \rt Y \rt X \rt N \rt 0$ be a MCM approximation of $N$, i.e., $X$ is maximal \CM \ and
$\injdim N  < \infty$. So we obtain $\Ext^i_A(M, X) \cong \Ext^i_A(M, N) = 0$ for all $i \gg 0$. We also note that $\injcurv X = \injcurv N$. Thus replacing $N$ by $X$ we can also assume that $N$ is a maximal  \CM \ $A$-module. Let $x \in \m$ be $A$-superficial.  Then $x$ is also $A$-regular. Set $B = A/(x)$. We note that $e(B) = e(A)$. As $M$ and $N$ are MCM we get that $x$ is both $M$ and $N$-regular.
The exact sequence $0 \rt N \rt N \rt N/xN \rt 0$ yields $\Ext_A^i(M, N/xN) = 0$ for $i \gg 0$. We have $\Ext_B^i(M/xM, N/xN) = \Ext_A^i(M, N/xN)$ for all $i \geq 0$, see
\cite[Lemma 2, section 18]{M}.
  As argued before we have $\curv_B(M/xM) = \curv_A(M)$. Also as $\Ext^{n+1}_A(k, N) \cong  \Ext^{n}_B(k, N/xN)$ for all $n \geq 0$, see \cite[Lemma 2, section 18]{M}, we obtain $\injcurv_B(N/xN) = \injcurv_A(N)$. Thus it suffices to prove the result for $B$. Iterating we can reduce to the case when $\dim A = 0$. After taking a high syzygy of $M$ we may assume that $\Ext_A^i(M, N) = 0$ for $i > 0$.

Let $$\mathbb{I}\colon \quad 0 \rt \omega^{r_0} \rt \omega^{r_1} \cdots \rt \omega^{r_n} \xrightarrow{\partial_n} \omega^{r_{n+1}}\rt \cdots$$ be a minimal injective resolution of $N$. Then note that $\Ext^n_A(k, N) = k^{r_n}$ for all $n \geq 0$. Set $N_{n+1} = \image \partial_n$. Set $M_n = \Omega^n(M)$.
Let $0 \rt N \rt \omega^{r_0} \rt N_1 \rt 0$.  So for all $n \geq 0$ we have an exact sequence
\[
0 \rt \Hom_A(M_n, N) \rt \Hom(M_n, \omega)^{r_0} \rt \Hom_A(M_n, N_1) \rt 0.
\]
So we have (see \ref{est})
\[
r_0 \ell(M_n^\dagger)  \geq \beta_n(M)(r_0(N) + r_0(N_1)).
\]
We have $\ell(M_n^\dagger) = \ell(M_n)$ and $r_0(N_1) = r_1(N) = r_1$.
So we have
$$ \ell(M_n) \geq \left(1 + r_1/r_0)\right)\beta_n(M). $$
 Similarly we obtain
$$ \ell(M_{n+1}) \geq \left(1 + \frac{r_1}{r_0}\right)\beta_{n + 1}(M). $$
 We have $\ell(M_n) + \ell(M_{n+1}) = e(A)\beta_n(M)$. Adding the two equations and dividing by $\beta_n(M)$ we obtain
 \[
 e(A) \geq \left(1 + \frac{r_1}{r_0}\right) \left( 1 + \frac{\beta_{n+1}(M)}{\beta_n(M)} \right).
 \]
 Taking lim sup and by \ref{rudin} we obtain
 \[
 e(A) \geq \left(1 + \frac{r_1}{r_0}\right) \left( 1 + \curv(M)\right).
 \]
 Replacing $N$ by $N_j$ we obtain
 \[
 e(A) \geq \left(1 + \frac{r_{j+1}}{r_j}\right) \left( 1 + \curv(M)\right).
 \]
 Taking lim sup and by \ref{rudin} we obtain
 \begin{equation*}
  e(A) \geq \left(1 + \injcurv(N)\right) \left( 1 + \curv(M)\right). \tag{**}
 \end{equation*}
 \[
 \]

 (a)  If $\projdim M $ or $\injdim N$ is finite then we have nothing to prove. So assume $\projdim M = \injdim N = \infty$.
 If $\min\{ \curv{M}, \injcurv(N) \} > \sqrt{e(A)} -1$ then by (**) we obtain $e(A) > e(A)$. This is a contradiction. The result follows.

(b) As $\projdim M = \injdim N = \infty$ we have $\curv(M) \geq 1$ and $\injcurv(N) \geq 1$. The result follows easily from (**).
\end{proof}

We now give
\begin{proof}[Proof of Corollary \ref{cor-first-s}]
(a) Suppose if possible $\projdim M = \projdim N = \infty$. Then by \ref{cor-first} we have $\curv(M) = \curv(k) > e(A)/2 -1$. This contradicts  the assertion of Theorem \ref{second}(1)(b). The result follows.

(b)Suppose if possible $\projdim M = \injdim N = \infty$. Then by \ref{cor-first} we have $\curv(M) = \curv(k) > e(A)/2 -1$. This contradicts  the assertion of Theorem \ref{second}(2)(b). The result follows.
\end{proof}
\section{Proof of Theorem \ref{third}}
We give
\begin{proof}[Proof of Theorem \ref{third}]
We note that $\Omega^{i_0}(M)$ is maximal \CM. We replace $M$ by $\Omega^{i_0}(M)$ and assume that $M$ is maximal \CM \ with infinite projective dimension. We first note that we may assume that the residue field of $A$ is infinite, see \ref{base}.  We reduce to the case when $\dim A = 0$ as follows. If $\dim A > 0$ then as the residue field of $A$ is infinite there exists an $A$-superficial element $x$. As $A$ is \CM \ we have $\depth A = \dim A > 0$. So $x$ is $A$-regular.
As $M$ is maximal \CM \ it is also $M$-regular. Set $B = A/(x)$.  We note that if $\mathbf{F}  $ is a minimal free $A$-resolution of $M$ then $\mathbf{F}\otimes_A B  $ is a minimal free $B$-resolution of $M/xM$. So $\curv_A(M) = \curv_B(M/xM)$.  We note that $x \in \m \setminus \m^2$. So by \ref{mod-x} we have $\liminf_n\sqrt[n]{\beta_n^B(k)} = \liminf_n\sqrt[n]{\beta_n^A(k)}$ and $\curv_B(k) = \curv_A(k)$.  Thus it suffices to prove the result for $B$. Iterating we may reduce to the case when $\dim A = 0$.

We have
$$ \frac{\beta_{1}(M)}{\beta_{0}(M)}  > \frac{e(A)}{1 + \curv(k)}  - 1.$$
 We want to prove that $\curv(M) \geq \liminf \sqrt[n]{\beta_n(k)}.$

Suppose if possible $\curv(M) < \liminf \sqrt[n]{\beta_n(k)}.$ Set $\curv(k) = \alpha$ and let \\ $\curv(M) = \theta $. We have $\theta < \liminf \sqrt[n]{\beta_n(k)} $.
  So by \ref{lim} we have
$\lim \beta_n(M)/\beta_n(k) = 0$. Set $X_i = \Omega^i(k)$.
By an argument identical to proof of Theorem \ref{first} we have
$$ e(A) \geq (1 + \beta_1(M)/\beta_0(M))(1+ \alpha).$$
So we get
$$ \frac{\beta_{1}(M)}{\beta_{0}(M)}  \leq \frac{e(A)}{1 + \curv(k)}  - 1.$$
This contradicts our hypotheses. The result follows.
\end{proof}

\begin{example}\label{cyclic}
We continue with our example in \ref{ex-1}. We prove that $I$ is an ideal of $A$ with $\mu(I) \geq 2$ then $\curv(A/I) = \curv(k) = h$.
To see this we note that
$$\mu(I)/1 = \beta_1(A/I)/\beta_0(A/I)  \geq 2.$$
However $$\frac{e(A)}{\curv{k} + 1} - 1 = 2(h+1)/(h+1) - 1 = 1.$$
It follows from Theorem \ref{third} that $\curv(A/I) \geq \liminf_n \sqrt[n]{\beta_n(k)} = h$.
But \\ $\curv(A/I) \leq \curv(k) = h$. The result follows.

Also $A/(x)$ has periodic resolution. Thus the result does not hold if $\mu(I) = 1$.
\end{example}

\end{document}